\documentclass[11pt,twoside,a4paper,reqno]{amsart}

\usepackage[a4paper,width=160mm,top=25mm,bottom=30mm]{geometry}
\usepackage[dvipsnames]{xcolor}

\usepackage{enumitem}

\usepackage[utf8]{inputenc}
\usepackage[T1]{fontenc}
\usepackage[english]{babel}

\usepackage{amsthm}

\theoremstyle{plain}
\newtheorem{theorem}{Theorem}

\newtheorem{proposition}[theorem]{Proposition}

\newtheorem{lemma}{Lemma}

\newtheorem{conjecture}{Conjecture}

\theoremstyle{remark}

\theoremstyle{definition}

\usepackage{amsmath}
\usepackage{amssymb}
\usepackage{dsfont}

\usepackage{mathtools}
\mathtoolsset{showonlyrefs}

\newcommand{\eps}{\varepsilon}
\newcommand{\lr}[1][]{\xleftrightarrow{#1}}
\newcommand{\abs}[1]{\left\lvert#1\right\rvert}

\usepackage{tabularx}
\usepackage{graphicx}

\usepackage{constants}
\newconstantfamily{small}{
symbol=c}

\title{Counting minimal cutsets and $p_c<1$}
\author{Philip Easo, Franco Severo, Vincent Tassion}
\date{}

\begin{document}

\begin{abstract}
    We prove two results concerning percolation on general graphs.
    \begin{itemize}[leftmargin=*, topsep=0.5em, itemsep=0.5em]
        \item We establish the converse of the classical Peierls argument: if the critical parameter for (uniform) percolation satisfies $p_c<1$, then the number of minimal cutsets of size $n$ separating a given vertex from infinity is bounded above exponentially in $n$. This resolves a conjecture of Babson and Benjamini from 1999.
        \item We prove that $p_c<1$ for every uniformly transient graph. This solves a problem raised by Duminil-Copin, Goswami, Raoufi, Severo and Yadin, and provides a new proof that $p_c<1$ for every transitive graph of superlinear growth.
    \end{itemize} 
\end{abstract}

\maketitle

\section{Main results}
\label{sec:main_res}

Let $G=(V,E)$ be an infinite, connected, locally finite graph. A set of edges $F\subset E$ is called a \emph{cutset} from a vertex $v$ to $\infty$ if $v$ belongs to a finite connected component of $(V,E\setminus F)$.
A cutset is called \emph{minimal} if no proper subset of it is a cutset. Let $\mathcal{Q}_n(v)$ be the set of minimal cutsets from $v$ to $\infty$ of cardinality $n$ and consider the  quantity
\begin{equation}
    q_n:=\sup_{v\in V}|\mathcal{Q}_n(v)|.
\end{equation}
Here $\abs{\emptyset}:=0$. We emphasize that $q_n=\infty$ is possible, for example for $G=\mathbb Z$ and $n=2$.  In this paper, we are interested in cases where the number of cutsets $q_n$ grows at most exponentially with $n$, and we define
\begin{equation}
  \label{eq:2}
  \kappa(G):=\sup_{n\ge1} q_n^{1/n}.
\end{equation}

Let $\mathrm P_p$ denote (Bernoulli bond) percolation of parameter $p \in [0,1]$ on $G$, where each edge is open with probability $p$ independently of the other edges. 
Consider the percolation probabilities 
$\theta_v(p):=\mathrm P_p(v\leftrightarrow \infty)$,
where $v\leftrightarrow \infty$ denotes the event that $v$ belongs to an infinite open connected component. 
We define the critical parameter for \emph{uniform} percolation as
 \begin{equation}
	p_c^*(G):=\inf\{p\in [0,1]:  \theta^*(p)>0\}, 
\end{equation}
where $\theta^*(p):=\inf_{v\in V} \theta_v(p)$.

By the classical Peierls argument \cite{Pei36} if $\kappa(G)<\infty$, then percolation on $G$ has a uniformly percolating phase in the sense that $p_c^*(G)<1$. Our first theorem establishes the converse.

\begin{theorem}\label{thm:perco->cutset}
For every infinite, connected, locally finite graph $G$ we have
  \begin{equation}
    p_c^*(G)<1 \iff \kappa(G)<\infty.
\end{equation}
\end{theorem}

Currently, the geometric condition $\kappa(G)<\infty$ is not well understood. Our second result gives a sufficient condition based on the simple random walk. Given a vertex $v$, let $\mathbb P_v$ be the law of a simple random walk $(X_t)_{t=0}^{\infty}$ on $G$ starting at $v$. We say that $G$ is uniformly transient if
\begin{equation}
  \label{eq:3}
  \inf_{v\in V}\left[d_v \cdot \mathbb P_v(\forall  t\ge 1: \ X_t\neq v)\right]>0,
\end{equation}
where $d_v$ denotes the degree of $v$.

\begin{theorem}\label{thm:transient->cutset}
 Let $G$ be an infinite, connected, locally finite graph.  If $G$ is uniformly transient, then $\kappa(G)<\infty$. 
\end{theorem}

\section{Consequences and comments}
\label{sec:cons-comm}

In this section, all graphs are assumed to be infinite, connected, and locally finite. Given a set of vertices $S$ in a graph $G=(V,E)$, we define the boundary $\partial S$ to be the set of all edges $\{u,v\} \in E$ such that $u \in S$ but $v \not\in S$, and we define the weight $\abs{S}_G := \sum_{u \in S} d_u$. The \emph{isoperimetric dimension} of $G$ is given by
\[
        \operatorname{Dim}(G) := \sup\left\{ d \geq 1 : \inf_{ \substack{ S \subseteq V \\ 0 < \abs{S} < \infty  } }  \frac{ \abs{\partial S} } { \abs{S}_G^{\frac{d-1}{d}} } > 0 \right\}.
\]

\begin{enumerate}[leftmargin=*, topsep=1.2em, itemsep=1.2em, label=\arabic*.]

\item We remark that the uniform critical parameter $p^*_c(G)$ slightly differs from the most classical (non-uniform) one given by $p_c(G):=\inf\{p\in [0,1]:  \theta(p)>0\}$, where $\theta(p):=\sup_{v\in V} \theta_v(p)$. However, these notions often coincide, such as for (quasi-)transitive graphs.
See the introduction of \cite{DGRSY20} for a survey of the rich history of the ``$p_c < 1$'' question and its place in statistical mechanics. Let us just recall that all of the results about percolation here can be translated into analogous statements about many other models, most notably the Ising model.

\item Duminil-Copin, Goswami, Raoufi, Severo, and Yadin proved that every quasi-transitive graph of superlinear growth satisfies $p_c < 1$ \cite{DGRSY20}. This had previously been a long-standing conjecture of Benjamini and Schramm \cite{BenSch96}. In fact, the authors of \cite{DGRSY20} established that $p^*_c<1$ for every (not necessarily transitive) bounded degree graph $G$ satisfying $\operatorname{Dim}(G) > 4$, and this was known to imply the conjecture about transitive graphs by the classical works of Gromov \cite{gromov1981groups} and Trofimov \cite{Tro84}.

\item Theorem \ref{thm:transient->cutset} establishes that $p^*_c < 1$ for every graph $G$ satisfying $\operatorname{Dim}(G) > 2$, since such graphs are uniformly transient (see e.g.~\cite[Theorem 6.41]{LP16}). We therefore obtain stronger results than \cite{DGRSY20}, through a completely new proof. Theorem \ref{thm:transient->cutset} fully realises the idea at the heart of \cite{DGRSY20} to exploit the transience of a simple random walk to prove $p_c < 1$. In particular, we resolve \cite[Problem 1.4]{DGRSY20}.

\item Our proofs of Theorems 1 and 2 can also be run on \emph{finite} graphs to establish the analogous results about \emph{giant} clusters. (See \cite{HT21} for background.) In this setting, to define $q_n$, one should instead count the number of minimal cutsets of cardinality $n$ from a vertex $v$ to another vertex $u$ (and take the supremum over all choices for distinct $u$ and $v$). The corresponding notion of uniform transience for a given family of finite graphs is that there exists a constant $C< \infty$ such that every graph $G=(V,E)$ in the family satisfies
\[
    \max_{u,v \in V} \mathcal R_G(u,v) \leq C,
\]
where $\mathcal R_G(u,v)$ denotes the effective resistance from $u$ to $v$ in the graph $G$.

\item Babson and Benjamini conjectured that $\kappa < \infty$ for every transitive\footnote{In fact, Babson and Benjamini originally made this conjecture in the case of Cayley graphs, and Benjamini later extended this conjecture to allow arbitrary transitive graphs.} graph of superlinear growth \cite{babson1999cut}. Notice that this purely geometric conjecture is a priori stronger than the above $p_c < 1$ conjecture of Benjamini and Schramm. Babson and Benjamini verified their conjecture in the special case of Cayley graphs of finitely presented groups by establishing that minimal cutsets in such graphs are coarsely connected. By \cite{timar2007cutsets,gromov1981groups,Tro84} (see also \cite[Lemma 2.1]{CMT21}), this extends to all transitive graphs satisfying $\operatorname{Dim}(G) < \infty$. Given these results, it suffices to show that $\kappa < \infty$ for every transitive graph satisfying $\operatorname{Dim}(G) = \infty$.
Theorem \ref{thm:transient->cutset} therefore resolves the $\kappa < \infty$ conjecture of Babson and Benjamini. (Alternatively, taking the results of \cite{DGRSY20} for granted, this conjecture follows from Theorem \ref{thm:perco->cutset}).

\item We establish the existence of a universal constant $\eps > 0$ such that every transitive graph $G$ satisfies $p_c = 1$ or $p_c \leq 1 - \eps$. When $G$ is recurrent, this follows from the proof of \cite[Theorem 1.7]{HT21}, and when $G$ is transient, this follows from our proof of Theorem \ref{thm:transient->cutset} because there exists a universal constant $c > 0$ such that a simple random walk in any transient transitive graph has probability at least $c$ never to return to where it started \cite[Corollary 1.3]{TT20}. Previous works had established this result if $\eps$ is allowed to depend on the degree of vertices in $G$ \cite[Theorem 1.7]{HT21}, or if we instead consider site percolation on a Cayley graph \cite{PS23,LMTT23}.
By the proof of Theorem~\ref{thm:perco->cutset}, we also obtain a universal constant $K<\infty$ such that $\kappa<K$ for every transitive graph of superlinear growth.

\item Much work has been motivated by a desire to find a sharp geometric criterion for a graph $G$ to satisfy ${p_c}<1$. Indeed, a well-known open conjecture of Benjamini and Schramm is that every (not necessarily transitive) graph $G$ with $\operatorname{Dim}(G) > 1$ satisfies ${p_c}<1$  \cite{BenSch96}. We were very surprised to find that the geometric criterion $\kappa < \infty$ (which is arguably simpler and more natural than the isoperimetric criterion) is not just sharp but \emph{exact}. Nevertheless, in light of Theorem \ref{thm:perco->cutset} and this conjecture of Benjamini and Schramm, we encourage the reader to investigate the following:
    \begin{conjecture}
        Every graph $G$ with $\operatorname{Dim}(G)>1$ satisfies $\kappa < \infty$. 
    \end{conjecture}
        
\end{enumerate}

The Peierls argument can be used to deduce results that are (a priori) much stronger than $p_c <1$. To explore these, it helps to consider the \emph{isoperimetric profile} $\psi$ of a graph $G=(V,E)$, given by
\[
        \psi(n) := \inf_{ \substack{ S \subseteq V \\ n \leq \abs{S}_G < \infty } }  \abs{\partial S}.
\]

\begin{enumerate}[leftmargin=*, topsep=1.2em, itemsep=1.2em,label =\arabic*.]
  \setcounter{enumi}{7}

\item Every graph $G=(V,E)$ satisfying $\kappa < \infty$ admits a \emph{strongly percolating} phase in the sense that for all $p \in (1-1/\kappa,1]$, there is a constant $c > 0$ such that
        \begin{equation} \begin{split}
                \mathbb P_p(S \not\leftrightarrow \infty) \leq e^{-c \psi(|S|)} \qquad &\text{for every finite set $S \subseteq V$};\\
                \mathbb P_p( n \leq |C_v| < \infty ) \leq e^{-c \psi(n) } \qquad &\text{for every $n \geq 1$ and $v \in V$}.
        \end{split}\end{equation}
        Thus our work resolves \cite[Problem 1.6]{DGRSY20} and implies that percolation on every transitive graph of superlinear growth has a strongly percolating phase. It remains an important open problem to establish that on these graphs, such bounds hold for \emph{all} $p \in (p_c,1]$. Indeed, this is the ``upper bound'' half of \cite[Conjecture 5.1]{HH19}.

\item Conversely, our proof of Theorem~\ref{thm:perco->cutset} (more precisely, Proposition~\ref{prop:fkg}) can be used to show that for every transitive graph $G=(V,E)$ and for every $p > p_c$, there is a constant $c > 0$ such that
        \[
                 \mathbb P_p( n \leq |C_v| < \infty ) \geq e^{-c \psi(n) } \qquad \text{for every $n \geq 1$ and $v \in V$}.
        \]
This establishes the ``lower bound'' half of \cite[Conjecture 5.1]{HH19}.

\item A major motivation for studying \emph{anchored} isoperimetric inequalities for graphs and manifolds is the belief that - unlike \emph{(uniform)} isoperimetric inequalities - anchored inequalities should typically be robust under small perturbations of the space \cite[Section 6]{BLS99}. We obtain the following concrete statement to this effect by combining Theorem \ref{thm:perco->cutset} with an argument of Pete \cite[Theorem 4.1]{Pet08b}: for every graph $G$ satisfying $p^*_c(G)<1$, there exists $\eps > 0$ such that if $G$ satisfies a $d$-dimensional anchored isoperimetric inequality for any $d \geq 1$ (or $f$-anchored isoperimetric inequality for any function $f$) then so does every infinite cluster formed by percolation of parameter $1-\eps$.

\item By combining the previous item with Theorem 2 and results of Thomassen \cite{Thom92} and Pemantle and Peres \cite{MR1395617}, we deduce that for every graph $G=(V,E)$ with $\operatorname{Dim}(G) > 2$, and for every probability measure $\mu$ on $(0,\infty)$, the random weighted network $(V,C)$ with $C = (C(e) : e \in E) \sim \mu^{\otimes E}$ is almost surely transient. (This was previously known if $\operatorname{Dim}(G) > 4$ \cite{Hut23}.) 

\item A standard analysis of \emph{Karger's algorithm} from computer science establishes that every finite graph $G=(V,E)$ with exactly $n$ vertices contains at most $n \choose 2$ \emph{minimum cuts}, i.e.\! sets of edges $F$ such that $(V,E\backslash F)$ is disconnected but there is no set of edges $F'$ with $|F'| < |F|$ such that $(V,E \backslash F')$ is also disconnected. In the same spirit, in the present paper, we design randomized algorithms to instead count \emph{minimal cutsets}. 

\end{enumerate}

\textbf{Acknowledgements: } We are very grateful to Benny Sudakov for telling us about Karger's algorithm from computer science. This seed is what prompted us to investigate probabilistic approaches to bounding the number of minimal cutsets, ultimately leading to the present work. We thank Itai Benjamini for bringing the $\kappa(G)<\infty$ question to our attention in the first place. We would also like to thank the anonymous referees for their useful comments on an earlier version of this paper.  PE is grateful for the hospitality provided by ETH Zurich during this project. This project was supported by the European Research Council (ERC) under the European Union’s Horizon 2020 research and innovation program (grant agreement No 851565). FS was supported by the ERC grant Vortex (No 101043450)

\section{Background and notation}
\label{sec:notation}

In this section, we fix $G=(V,E)$ a locally finite, connected graph.

\subsection*{Paths and connectivity}
\label{sec:graph-theory}
Let  $S\subset V$, $u,v\in S$.  A {path from $u$ to $v$ in $S$}   is a finite  sequence  $\gamma=(\gamma_0,\gamma_1,\ldots,\gamma_\ell)$   of {distinct} vertices of $S$ such that $\gamma_0=u$, $\gamma_\ell=v$ and $\{\gamma_{i-1},\gamma_{i}\}\in E$ for every $i\in\{1,\ldots,\ell\}$.  When such a path exists, we say that {$u$ is connected to $v$ in $S$}. By extension, a set $A$ is said to be connected to a set $B$ in $S$ if there exists a   vertex of $A$ that is connected to a vertex of $B$ in $S$. 
A  {path from $u$ to $\infty$ in $S$} is an infinite sequence of distinct vertices $\gamma_0,\gamma_1,\ldots$ in $S$  such that $\gamma_0=u$ and   $\{\gamma_{i-1},\gamma_{i}\}\in E$ for every $i\in\{1,2,\ldots\}$.   When such a path exists, we say that $u$ is connected to $\infty$ in $S$.

\subsection*{Exposed boundary}
 Let $S\subset V$ be a finite set. The exposed boundary of $S$ is the set $\partial_\infty S$  of  all the edges $\{u,v\}$ such that  $u\in S$ and $v$ is connected to $\infty$ in $V\setminus S$. Notice that  the exposed boundary is a subset of the standard boundary defined at the beginning of Section~\ref{sec:cons-comm}: for every finite set $S\subset V$, we have $\partial_\infty S\subset \partial S$.
  
\subsection*{Percolation configurations}
\label{sec:backgr-bern-perc}
An element  $\omega\in \{0,1\}^E$ is called a  percolation configuration. Given such a configuration, an edge $e\in E$ is said to be open if $\omega(e)=1$ and closed if $\omega(e)=0$. By extension, a path is said to be open if all its edges are open. The cluster of a vertex  $u\in V$  is the connected component of $u$ in the graph $(V,\{e\in E\: : \: \omega(e)=1\})$.

\subsection*{Percolation events} A measurable subset $A\subset \{0,1\}^E$ is called a percolation event.  Given $S\subset V$ and $u,v\in S$, we denote by $u\xleftrightarrow{S} v$ the event that there exists an open path from $u$ to $v$ in $S$, and simply write $u\lr v$ when $S=V$. Finally, $u\lr\infty$ denotes the event that there exists an open path from $u$ to $\infty$ in $V$.

\subsection*{Percolation measures}
A percolation measure on $G$ is a probability measure on the product space $\{0,1\}^E$.  For $p\in [0,1]$, we denote by $\mathrm P_p$ the standard Bernoulli percolation measure, under which each edge is open with probability $p$ independently of the other edges.

\subsection*{\bf Positive association} A percolation  event $\mathcal E$  is called increasing if for all percolation configurations $\omega,\xi$ satisfying  $\omega\le \xi$  for the standard product (partial) ordering, we have  $\omega\in \mathcal E\implies \xi\in \mathcal E$. Typical examples  of increasing events are the connection events (such as $u\xleftrightarrow{S} v$) introduced above. A percolation measure $\mathrm P$ is said to be {positively associated} if 
\begin{equation}
  \label{eq:14}
  \mathrm P[\mathcal E\cap \mathcal F] \ge \mathrm P[\mathcal E] \mathrm P[\mathcal F] 
\end{equation}
for all increasing events $\mathcal E,\mathcal F$.  This property is often  referred to as the FKG inequality. We will use that  Bernoulli percolation $\mathrm P_p$ is positively associated (for every fixed $p\in [0,1]$) as established by  Harris \cite{Har60}.

\section{Exposed boundaries and cutsets}
\label{sec:expos-bound-cuts}

In this section, we fix $G=(V,E)$ an infinite, connected, locally finite graph. In our paper, we will use that minimal cutsets can be obtained by considering the exposed boundary of finite connected sets. In this section, we  recall some well-known  facts relating the two notions. The first elementary result  is that the exposed boundary of a finite connected set is a minimal cutset.

\begin{lemma}\label{lem:1}
  Let $S\subset V$ be a finite connected set. For every $u\in S$, $\partial_\infty S$ is a minimal cutset from $u$ to $\infty$.
\end{lemma}
\begin{proof}
 Any path from $u$ to $\infty$ in $V$ must traverse an edge in $\partial_\infty S$ (consider the last edge traversed by this path intersecting $S$). Therefore, $\partial_\infty S$ is a cutset from $u$ to $\infty$.   To prove that it is minimal, consider an edge $e \in \partial_\infty S$. Since $S$ is connected, there exists a path from $u$ to an endpoint of $e$ in $S$ and by definition of the exposed boundary, there must exist a path from the other endpoint of $e$ to~$\infty$ in $V\setminus S$. The concatenation of these two paths with $e$ connects $u$ to $\infty$ without using any  edges of $\partial_\infty S$ other than $e$. Hence $\partial_\infty S\setminus\{e\}$ is not a cutset from $u$ to $\infty$.  
\end{proof}

The second elementary result identifies the exposed boundary under some simple conditions.

\begin{lemma}\label{lem:2}
  Let $u\in V$, let $\Pi$ be a minimal cutset from $u$ to $\infty$. Let $A$ be the connected component of $u$  in $(V,E\setminus \Pi)$ and $B=\{e\cap A,\, e\in \Pi\}$  be the set of inner vertices of $\Pi$. For every set $S$ of vertices, we have
  \begin{equation}
    \label{eq:1}
    (B \subset S \subset A) \implies (\partial_\infty S= \Pi).
  \end{equation}
  \end{lemma}
\begin{proof}
    Since $A$ is a maximal connected set in $(V,E\setminus \Pi)$, all the edges at the boundary of $A$ belong to $\Pi$, and therefore 
    $\partial _\infty A\subset \partial A\subset \Pi$.
    By Lemma~\ref{lem:1}, $\partial_\infty A$ is a cutset from $u$ to $\infty$, hence, by the minimality of $\Pi$,  the two inclusions above must be equalities:
    \begin{equation}
      \label{eq:5}
      \partial _\infty A =\partial A= \Pi.
    \end{equation}
    Now, let  $S$ be a set satisfying $B\subset S\subset A$. Let $e\in \Pi$. Since $e\in \partial_\infty A$, one endpoint of $e$ must belong to $B$ and the other endpoint is connected to $\infty$ in $V\setminus A$.  Therefore, by hypothesis, one endpoint of $e$ belongs to $S$ and the other endpoint is connected to $\infty$ in $V\setminus S$. This proves the inclusion
    \begin{equation}
      \label{eq:6}
      \Pi\subset \partial_\infty S.
    \end{equation}
    Let $e \in \partial_{\infty}S$. Let $u$ be the endpoint of $e$ in $S$, and let $v$ be the endpoint of $e$ connected to $\infty$ in $V \backslash S$. Then, by hypothesis, $u \in A$ and $v$ is connected to $\infty$ in $V \backslash B$. Since $\Pi = \partial A$, every edge in $\Pi$ intersects $A$ and hence intersects $B$. Therefore, there must exist an infinite path starting at $v$ in the subgraph $(V,E \backslash \Pi)$. In particular, $v \not\in A$, and hence $e \in \partial A = \Pi$. This proves that the inclusion above must be an equality.

  \end{proof}

\section{Full connectivity via positive association}
\label{sec:FKG-lemma}
In this section, we consider the following problem: Let $B$ be a finite set in a graph, and $\mathrm P$ be a percolation measure. What is the probability that all the vertices of $B$ are all connected to each other? Or, in other words, what is the probability that all the vertices of $B$ lie in the same cluster? We prove that this probability is at least exponential in the size of $B$ when the measure is positively associated, and the probability for a point to be connected to $B$ is uniformly lower bounded. This result, formally stated below, will allow us to construct random sets with a prescribed boundary.

\begin{proposition} \label{prop:fkg}
  Let $G=(V,E)$ be a finite, connected graph. Let $\mathrm P$ be a positively associated percolation  measure on $G$. Let $B \subset  V$, let $\theta,p \in (0,1]$ and suppose that  $\mathrm P(u \leftrightarrow B) \geq \theta$ for every $u \in V$, and $\mathrm P(e \text{ is open}) \geq p$ for every $e \in E$.
Then for every $o \in V$,
    \[
        \mathrm P\left( \bigcap_{b \in B} \{ o \leftrightarrow b \}  \right) \geq c^{|B|},
      \]
      where  $c := \left(\frac{p\theta}{2}\right)^{3/\theta}$.
\end{proposition}

\begin{proof}
Say that a finite sequence of vertices $x_1,\dots,x_k$ is \emph{chained} if $x_1 = o$ and for all $i \in \{2,\dots,k\}$,
\begin{equation}
  \label{eq:17}\tag{\bf P1}
  \frac{p\theta}{2} \leq \mathrm P\left( x_i \leftrightarrow \{x_1,\dots,x_{i-1}\} \right) \leq \frac{\theta}{2}.
\end{equation}
Since there exists at least one chained sequence (take $k = 1$) and $V$ is finite, there must exist a chained sequence $x_1,\dots,x_k$ that is \emph{maximal} in the sense that for every vertex $x_{k+1}$, the sequence $x_1,\dots,x_{k+1}$ is not chained. Fix a maximal chained sequence $x_1,\dots,x_{k}$, and let $X := \{x_1,\dots,x_{k}\}$. We claim that, in addition to \eqref{eq:17}, this sequence satisfies the following two properties, where $n:= \abs{B}$:
\begin{align}
  \label{eq:19}
  \tag{\bf P2} &\forall u\in V \quad \mathrm P(u\lr X)\ge \frac{\theta}2,\\
   \tag{\bf P3} &k \leq \frac{2n}{\theta}.\label{eq:20}
\end{align}

To prove \eqref{eq:19}, consider the set of vertices $W\subset V$  that are connected to $X$ with probability at least $\theta/2$ and suppose for contradiction that $W\neq V$. Since $W$ is non empty (because $X\subset V$) and $G$ is connected, we can consider an edge $\{u,v\}$ such that $u\in W$ and $v\notin W$.     
 By positive association,
\[
    \theta/2>\mathrm P( v \leftrightarrow X ) \geq \mathrm P(\{u,v\} \text{ is open}) \cdot \mathrm P(u \leftrightarrow X) \geq \frac{p \theta}{2}.
\]
In particular, $x_1,\dots,x_{k},v$ is a chained sequence, contradicting the maximality of $x_1,\dots,x_k$.

We now prove \eqref{eq:20}. To this aim, for each $i \in \{1,\dots,k\}$, let $N_i$ denote the number of clusters that  intersect both $\{x_1,\dots,x_i\}$ and $B$.
For every $i\in \{2,\dots,k\}$,  the increment $N_{i}-N_{i-1}$ is equal to $1$ if $x_i$ is connected to $B$ but not to the previous points $\{x_1,\ldots,x_{i-1}\}$, and it is equal to $0$ otherwise. Therefore, for every $i\in\{2,\ldots,k\}$, we have the deterministic inequality  
\begin{equation}
  \label{eq:15}
  N_{i}-N_{i-1} \ge \mathbf 1_{x_i \lr B} - \mathbf 1_{x_i\lr\{x_1,\ldots, x_{i-1}\}}.
\end{equation}
Taking the expectation, using our hypothesis and \eqref{eq:17}, for every $i\in\{2,\ldots, k\}$, we get
\begin{equation}
  \label{eq:16}
  \mathrm  E(N_i) - \mathrm E(N_{i-1}) \ge \underbrace{\mathrm P(x_i \lr B)}_{\ge \theta}- \underbrace{\mathrm P(x_i\lr\{x_1,\ldots, x_{i-1}\})}_{\le\theta/2}\ge \theta/2.
\end{equation}
Summing over $i\in\{2,\ldots, k\}$ and using  $\mathrm E(N_1)=\mathrm P(x_1\lr B)\ge \theta\ge \theta/2$, we get  $\mathrm E[N_k] \geq \frac{\theta}{2}k$. Since $N_k$ is deterministically bounded above by $|B| = n$, this concludes the proof of \eqref{eq:20}.

We now explain how the three properties above of the chained sequence imply the desired lower bound in the proposition. First, we estimate the event that all the vertices of $X$ are connected to $o$: By  \eqref{eq:17}, \eqref{eq:20}  and  positive  association, we have 
\begin{equation} \label{eq:internally_connected}
    \mathrm P\left( \bigcap_{u\in X}\{o\lr u \}\right) \geq \prod_{i=2}^k \mathrm P\left( x_i \leftrightarrow \{x_1,\dots,x_{i-1} \}\right) \geq \left( \frac{p\theta}{2} \right)^{k-1} \geq \left( \frac{p \theta}{2} \right)^{\frac{2n}{\theta}}.
\end{equation}

Second, we estimate the event that all the vertices of $B$ are connected to $X$: By \eqref{eq:19} and  positive association, we have 
\[
    \mathrm P \left( \bigcap_{b\in B} \{ b \leftrightarrow X\} \right) \geq \left( \frac{\theta}{2} \right)^n.
\]
If all the vertices of $X$ are connected to $o$ and all the vertices of $B$ are connected to $X$, then all the vertices of $B$ are connected to $o$. Hence, by the two displayed equations above and positive association, we obtain
\[
    \mathrm P\left( \bigcap_{b\in B} \{ o \leftrightarrow b \}  \right) \geq\mathrm P\left( \bigcap_{u\in X}\{o\lr u \}\right)  \cdot \mathrm P\left( \bigcap_{b \in B} \{ b \leftrightarrow X \}  \right) \geq \left( \frac{p\theta}{2} \right)^{\frac{2n}{\theta}} \cdot \left( \frac{\theta}{2} \right)^n \geq c^n,
\]
where $c := \left(\frac{p\theta}{2}\right)^{3/\theta}$.
\end{proof}

\section{Proof of Theorem \ref{thm:perco->cutset}}
\label{sec:perco->cutset}

Let $G=(V,E)$ be an infinite, connected, locally finite graph. In this section, we prove Theorem~\ref{thm:perco->cutset}, in the following form. 
\begin{equation}
  \label{eq:4}
\big  (\exists p<1 \ \exists \theta>0 \ \forall u \in V \quad\mathrm P_p(u\leftrightarrow \infty)\ge \theta \big )\iff \big (\exists K<\infty \ \forall u\in V\ \forall n\ge1\quad |\mathcal Q_n(u)|\le K^n\big). 
\end{equation}
 
The implication $\Leftarrow$ is well-known, and follows from the Peierls argument \cite[Theorem 4.11]{Ben13}, which we now recall for completeness. Let $u\in V$. If the cluster of $u$ is finite, then by Lemma~\ref{lem:1}, its exposed boundary is a finite minimal cutset from $u$ to $\infty$, and all its edges are closed. Hence, by the union bound, for every $p\in [0,1]$ we have 
\begin{equation}
  \label{eq:12}
  \mathrm P_p(|C_u|<\infty)\le \sum_{n\ge 1} q_n (1-p)^n. 
\end{equation}
If $q_n\le K^n$ for some constant $K<\infty$,  then the right hand side above converges to $0$ as $p$ tends to $1$. Since the bound is uniform in $u$,  there exists $p<1$ such that
\begin{equation}
  \label{eq:13}
  \forall u\in V \quad \mathrm P_p(u\leftrightarrow \infty)\ge 1/2.
\end{equation}

We now prove the implication $\Rightarrow$. Fix $\theta,p\in(0,1)$ such that $\mathrm P_{p}(u \leftrightarrow \infty) \geq \theta$ for every $u\in V$. Fix $o\in V$ and $n\ge 1$. Writing  $C$ for the cluster of $o$, we show that  for every minimal cutset $\Pi$ from $o$ to $\infty$ with $\abs{\Pi} = n$,
\begin{equation}
  \label{eq:8}
   \mathrm P_p (\partial_{\infty} C = \Pi ) \geq 1/K^n,
\end{equation}
where $K=K(p,\theta)\in (0,\infty)$ is a finite  constant depending on $p$ and $\theta$ only (in particular it does not depend on the chosen vertex $o$). This concludes the proof since
\begin{equation}
  \label{eq:9}
1\ge    \sum_{ \Pi \in \mathcal Q_n(o) } \mathrm P_p(\partial_{\infty} C = \Pi ) \overset{~\eqref{eq:8}}\ge  |\mathcal{Q}_n(o)|  /K^n. 
\end{equation}

Let us now prove the lower bound \eqref{eq:8}. As in Lemma~\ref{lem:2}, let $A$ be the connected component of $o$ in $(V,E\setminus \Pi)$ and $B$ the set of inner vertices of $\Pi$. Since any infinite open path from a vertex $u\in A$  must intersect $B$ before exiting $A$, the  hypothesis $\mathrm P_p(u\leftrightarrow \infty)\ge \theta$  implies
\begin{equation}
  \label{eq:7}
  \forall u\in A \quad \mathrm P_p(u\xleftrightarrow{A} B)\ge \theta.
\end{equation}
Let $\mathcal E$ be the  event  that every vertex in $B$ is connected to $o$ by an open path in $A$.    By Proposition~\ref{prop:fkg} applied to the finite subgraph of $G$ induced by $A$, we have $\mathrm P_{p}(\mathcal E) \geq c^n$, where $c=(p\theta/2)^{3/\theta}>0$. Let $\mathcal F$ be the event that all the edges of $\Pi$ are closed. By independence, we have
\begin{equation}
  \label{eq:10}
  \mathrm P_p(\mathcal E\cap \mathcal F)= \mathrm P_p(\mathcal E)\mathrm P_p(\mathcal F)\ge c^n (1-p)^n.
\end{equation}
If the event $\mathcal E\cap \mathcal F$ occurs, then the cluster $C$ of $o$ satisfies $B\subset C\subset A$. Hence, by Lemma~\ref{lem:2} we must have $\partial_\infty C=\Pi$. This concludes that
\begin{equation}
  \label{eq:11}
  \mathrm P_p(\partial_{\infty} C=\Pi)\ge \mathrm P_p(\mathcal E\cap \mathcal F)\ge c^n(1-p)^n,
\end{equation}
which establishes the desired lower bound~\eqref{eq:8} with $K=\frac1{c(1-p)}=\frac1{(p\theta/2)^{3/\theta}(1-p)}$.

\section{A covering lemma for Markov chains}\label{sec:RW_lemma}

In this section, we give conditions under which a killed Markov chain survives long enough to visit every state and then return to its initial state\footnote{In fact, we lower bound the probability that this occurs in $\leq 2n-2$ steps (which is optimal), where $n$ is the number of states. Contrast this with \cite{BGM13,DK21}, both called \emph{Linear cover time is exponential unlikely}; we give conditions under which linear cover time is exponentially \emph{likely}.}. We will apply this in the next section to prove Theorem \ref{thm:transient->cutset}. Here $[n]$ denotes the set $\{1,\dots,n\}$.

\begin{lemma}\label{lem:covering}
Let $n \geq 1$. Let $P=(p_{i,j})_{i,j \in [n]}$ be a symmetric matrix of non-negative entries such that\footnote{We can think of $P$ as the transition matrix of a Markov chain which is killed at $i$ with probability $1-\sum_{j} p(i,j)$.} $\sum_{j \in [n]} p(i,j) \leq 1$ for all $i \in [n]$. Let $\Gamma$ be the set of all sequences $\gamma=(\gamma_0,\gamma_1,\ldots,\gamma_k)$ in $[n]$ (for any $k \geq 1$) with $\gamma_0=1$ such that the unique element $i\in [k]$ satisfying both $\gamma_i=1$ and $\{\gamma_0,\gamma_1,\ldots,\gamma_i\} = [n]$ is $i=k$. For every such sequence $\gamma$, define
\[
          p(\gamma) := \prod_{i=1}^k p\left(\gamma_{i-1},\gamma_i\right).
\]
For each $\eps > 0$, if every non-empty proper subset $I$ of $[n]$ satisfies
\begin{equation} \label{eq:matrix_has_good_crossing}
          \sum_{i \in I} \sum_{j \in [n] \backslash I} p(i,j) \geq \eps,
\end{equation}
then $\delta :=\frac{\eps^2}{16e^2}$ satisfies
\[
          \sum_{\gamma \in \Gamma} p(\gamma) \geq \delta^n.
\]
\end{lemma}

\begin{proof}
Let $e_1,\dots,e_{2n-2} \in [n]^2 \sqcup \{\emptyset\}$ be an iid sequence of random variables such that for all $u,v \in [n]$,
\[
          \mathbb P\left( e_1 = (u,v) \right) =\frac{p(u,v)}{n}.
\]
Such random variables exist because these probabilities sum to at most 1. Let $H$ be the \emph{undirected} multigraph with vertex set $[n]$ and edges $e_1,\dots,e_{2n-2}$. Even though $[n]^2$ consists of \emph{ordered} pairs, we think of each $e_i \in [n]^2$ as encoding an \emph{undirected} edge, loops allowed. (When $e_i = \emptyset$, we simply do not include an edge.)

Consider the iid spanning subgraphs $H_1$ and $H_2$ of $H$ that contain only the edges $e_1,\dots,e_{n-1}$ and $e_{n},\dots,e_{2n-2}$, respectively. We will lower bound the probability that each of these graphs is connected. Consider any $k \in [n-1]$. Suppose that we are given all of the connected components $C_1,\dots,C_r$ of the spanning subgraph of $H$ that contains only the edges $e_1,\dots,e_{k-1}$. If $r \geq 2$, then the conditional probability that $e_{k}$ connects two of these components is
\[
          \sum_{z=1}^r \sum_{ i\in C_z} \sum_{ j \in [n] \backslash C_z} \frac{p(i,j)}{n} \overset{\eqref{eq:matrix_has_good_crossing}}\geq \frac{r \eps}{n}.
\]
Therefore by induction on $k$, and by using the elementary bound $\frac{n^n}{n!} \leq e^n$ in the third inequality,
\begin{equation} \label{eq:H_1_is_connected}
          \mathbb P\left( H_1 \text{ is connected} \right) \geq \prod_{r=2}^{n} \frac{r \eps}{n} \geq \frac{n! \cdot \eps^n}{n^{n} } \geq \frac{\eps^n}{e^{n}}.
\end{equation}

Let $\gamma=(\gamma_0,\gamma_1,\ldots,\gamma_k)$ be a sequence in $\Gamma$. Say that $\gamma$ is \emph{present} if there exists an injection $\sigma : [k] \to [2n-2]$ such that for every $i \in [k]$, we have $e_{\sigma(i)} = (\gamma_{i-1},\gamma_i)$ or $(\gamma_i,\gamma_{i-1})$. Assume that $k \leq 2n-2$, and note that $\gamma$ cannot be present otherwise. There are at most $(2n-2)^k$ choices of $\sigma$, and given $\sigma$, for each $i$, the probability that $e_{\sigma(i)} = (\gamma_{i-1},\gamma_i)$ is the same as the probability that $e_{\sigma(i)}=(\gamma_i,\gamma_{i-1})$, both given by $\frac{1}{n}p( \gamma_{i-1},\gamma_i ) = \frac{1}{n}p( \gamma_i,\gamma_{i-1} )$. So by a union bound,
\begin{equation} \label{eq:u_unlikely_to_be_present}
          \mathbb P\left( \gamma \text{ is present} \right) \leq (2n-2)^k \prod_{i=1}^k \frac{2 }{n} p\left( \gamma_{i-1},\gamma_i \right) \leq 4^k p(\gamma) \leq 4^{2n}p(\gamma).
\end{equation}
On the other hand, when $H_1$ is connected and $H_2$ is connected, then \emph{some} $\gamma \in \Gamma$ must be present in $H$ because every multigraph that contains two edge-disjoint spanning trees must also contain a spanning subgraph that is connected and Eulerian \cite[Corollary 2.3A]{Cat92}\footnote{This general fact can be proved directly as follows: Let $E$ be the set of edges in the multigraph, and let $T_1$ and $T_2$ be the two trees. Let $o_1,\dots,o_{k}$ be the vertices that have odd degree in $T_1$. Since the sum of the degrees of all of the vertices in a given graph is always even, we can write $k=2l$ for some non-negative integer $l$. For each $i \in [l]$, pick a path $P_i$ in $T_2$ from $o_{2i-1}$ to $o_{2i}$. Then, viewing $T_1$ and each $P_i$ as elements of the $\mathbb Z/2 \mathbb Z$-vector space $\{0,1\}^{E}$, the required subgraph is given by $T_1 + P_1 + \dots + P_l$.}.
Thanks to \eqref{eq:H_1_is_connected}, this occurs with probability at least $\eps^{2n}/e^{2n}$. So by a union bound,
\begin{equation} \label{eq:some_u_must_be_present}
         \frac{\eps^{2n}}{e^{2n}} \leq \sum_{\gamma \in \Gamma} \mathbb P(\gamma \text{ is present}) \leq 4^{2n} \sum_{\gamma \in \Gamma} p(\gamma).
\end{equation}
The conclusion follows by rearranging. 
\end{proof}

\section{Proof of Theorem \ref{thm:transient->cutset}}
\label{sec:transient->cutset_RW}

Let $G=(V,E)$ be an infinite, connected, locally finite graph such that for some constant $\eps > 0$, for every vertex $v \in V$, the simple random walk $(X_t)_{t=0}^{\infty}$ on $G$ started at $v$ satisfies
\[
           d_v \mathbb P_v(\forall  t\ge 1: \ X_t\neq v) \geq \eps.
\]
Let $G'=(V',E')$ be the graph\footnote{This construction is a technicality that is only necessary if $G$ has unbounded vertex degrees.} obtained from $G$ by replacing each edge by a path of length 2. View $V$ as a subset of $V'$, and let $m: E \to V'$ map each edge to its midpoint. Let $\mathbb P_u'$ be the law of simple random walk in $G'$ started from a given vertex $u$, and let $\tau := \sup\{ t \geq 0 : X_t = X_0 \}$. We claim that for all $z \in V'$,
\begin{equation} \label{eq:uniform_transience_hypothesis}
          d_z \mathbb P_z'(\tau =0 ) \geq \eps_1:=\frac{2\eps}{4+\eps}.
\end{equation}
This is trivial when $z \in V$, even with $\eps_1 = \eps/2$, because simple random walk on $G'$ induces lazy simple random walk on $G$. Otherwise, when $z = m(\{u,v\})$ for some $\{u,v\} \in E$, this follows from the corresponding bounds for $u$ and $v$ by rearranging the following elementary calculation, where $\ell_x := |\{t \geq 0 : X_t = x\}|$:
\[\begin{split}
    \frac{1}{\mathbb P_z'(\tau = 0)} = \sum_{n \geq 0}\mathbb P_z'(\tau > 0)^n = \mathbb E_z'[\ell_z] &= \sum_{t \geq 0} \mathbb P_z'(X_t = z)\\ &= 1 + \sum_{t \geq 1}\left[
    \mathbb P_z'(X_{t-1} = u) \cdot \frac{1}{d_u} + \mathbb P_z'(X_{t-1} = v) \cdot \frac{1}{d_v}\right]\\
    &= 1 +  \frac{\mathbb E_z'[\ell_u]}{d_u} + \frac{\mathbb E_z'[\ell_v]}{d_v} \\&\leq  1 +  \frac{\mathbb E_u'[\ell_u]}{d_u} + \frac{\mathbb E_v'[\ell_v]}{d_v} = 1 + \frac{1}{d_u \mathbb P_u'(\tau = 0)} + \frac{1}{d_v \mathbb P_v'(\tau = 0)}.
\end{split}\]

Let $C := \{X_t : 0 \leq t \leq \tau\}$ and $\partial := \{ e \cap C : e \in \partial_\infty C \}$. Fix $o \in V$, and pick a neighbour $o' \in m(E)$ of $o$ in $G'$. Fix a finite minimal cutset $\Pi$ from $o$ to $\infty$ in $G$, and set $n := \abs{\Pi}$. We will show that for some finite constant $K=K(\eps) \in (0,\infty)$ depending only on $\eps$,
\begin{equation} \label{eq:lower_bound_on_particular_random_walk_output}
          \mathbb P_{o'}' \left( \partial = m(\Pi) \right) \geq 1/K^n.
\end{equation}
This implies that $\kappa(G) < \infty$ because for all $o \in V$ and $n \geq 1$,
\[
          1 \geq \sum_{\Pi \in \mathcal Q_n(o)} \mathbb P_{o'}' \left( \partial = m(\Pi) \right) \overset{\eqref{eq:lower_bound_on_particular_random_walk_output}}\geq \abs{\mathcal Q_n(o)}/K^n.
\]

Let $A$ be the connected component of $o$ in $(V,E \backslash \Pi)$, let $U := m(\Pi) \cup \{o'\}$, and let $I := A \cup m( \{e \in E : e \subset A\} )$. For all $u,v \in U \cup I$, let
\[
          p(u,v) := \mathbb P_{u}'\left( \exists t \geq 1 : X_1,\dots,X_{t-1} \in I \backslash \{u\} \text{ and } X_t = v  \right).
\]
Extend this to sets of vertices by $p(L,R) := \sum_{u \in L; v \in R} p(u,v)$, and similarly, $p(u,L) := p(\{u\},L)$ and $p(L,u) := p(L,\{u\})$. We would like to apply Lemma \ref{lem:covering} to the matrix $P:=(p(u,v))_{u,v \in U}$. By time-reversing trajectories, we have $p(u,v) = p(v,u)$ whenever $d_u = d_v$, which is for example the case when $u,v \in U$. So $P$ is symmetric, and clearly the entries of $P$ are non-negative and sum to at most 1 along each row. We claim that for every non-trivial partition $U = L \sqcup R$,
\begin{equation} \label{eq:good_crossing_of_partition}
          p(L,R) \geq \eps_2 := \eps_1^2/64.
\end{equation}
Indeed, for each $x \in U \cup I$, consider the function (the unit voltage)
\[
          F(x) := \mathbb P_x'\left( \exists t \geq 0 :  X_0,\dots,X_{t-1} \not\in L \text{ and }X_t \in R \right).
\]
Given $u \in L$, there exists\footnote{If $u = o'$, take $x := o$. If $u = m(e)$ where $e \in \Pi$, take $x$ where $\{x\} = e \cap A$, which exists by Lemma \ref{lem:2}.} $x \in A$ such that $\{u,x\} \in E'$, and if $F(x) \geq 1/2$, then we are done because
\[
          p(u,R) \geq \mathbb P_u'(X_1 = x) \cdot F(x) \geq 1/2 \cdot 1/2 \geq \eps_2.
\]
In particular, we may assume that there exists $x \in A$ with $F(x) < 1/2$. By a similar argument, we may assume that there exists $y \in A$ with $F(y) > 1/2$. Since $A$ is connected in $G$, we can therefore find $\{x',y'\} \in E$ satisfying $F(x') \leq 1/2 \leq F(y')$. Let $z := m(\{x',y'\})$, which has degree 2. Note that
\[
          F(z) \geq \mathbb P_z'(X_1 = y') \cdot F(y') \geq 1/2 \cdot 1/2,
\]
and by a union bound,
\[
          F(z) \leq \sum_{n = 0}^{\infty} \mathbb P_z'(\tau > 0)^n p(z,R) = \frac{p(z,R)}{\mathbb P_z'(\tau = 0)} \overset{\eqref{eq:uniform_transience_hypothesis}}\leq \frac{p(z,R)}{\eps_1/2}.
\]
So by rearranging, $p(z,R) \geq \eps_1/8$. By a similar argument  (i.e.\! by replacing $F$ by $1-F$, which switches the roles of $L$ and $R$, and by recalling that $p(L,z)=p(z,L)$), we deduce that $p(L,z) \geq \eps_1/8$. Now \eqref{eq:good_crossing_of_partition} follows because $p(L,R) \geq p(L,z) p(z,R)$.

Therefore by Lemma \ref{lem:covering}, the event $\mathcal E$ that the random walk visits every vertex in $U$ then returns to $o'$ before exiting $U \cup I$ satisfies
\begin{equation} \label{eq:cover_without_bad_crossing}
          \mathbb P_{o'}'\left( \mathcal E \right) \geq \eps_3^{\abs{U}} \geq \eps_3^{n+1}
\end{equation}
for some constant $\eps_3 > 0$ depending only on $\eps_2$. So by Lemma \ref{lem:2} and the strong Markov property,
\begin{equation} \label{eq:sufficient_to_have_good_output}
           \mathbb P_{o'}'\left( \partial = m\left(\Pi \right) \right) \geq \mathbb P_{o'}'\left( \mathcal E \right) \cdot \mathbb P_{o'}'\left( \tau = 0 \right) \overset{\eqref{eq:uniform_transience_hypothesis},\eqref{eq:cover_without_bad_crossing}}\geq \eps_3^{n+1} \cdot \eps_1/2.
\end{equation}
By expanding the definitions of $\eps_1,\eps_2,\eps_3$ we deduce that \eqref{eq:lower_bound_on_particular_random_walk_output} holds with $K := 2^{20}/\eps^5$.

\section{Alternative proof of Theorem \ref{thm:transient->cutset} using the Gaussian free field}
\label{sec:transient->cutset_GFF}

Here we sketch an alternative, slightly less elementary proof of Theorem~\ref{thm:transient->cutset} along the lines of the proof of Theorem~\ref{thm:perco->cutset}. Let $G$ be an infinite, connected, locally finite graph that is uniformly transient. Consider the graph $\tilde{G}=(\tilde{V},\tilde{E})$ obtained by replacing each edge by a path of length $3$. Similarly to the proof in Section~\ref{sec:transient->cutset_RW}, one can prove that $\tilde{G}$ is also uniformly transient. 
Let $\varphi\in \mathbb{R}^{\tilde{V}}$ with law $\mathbb P$ be the (centered) \emph{Gaussian free field (GFF)} on $\tilde G$ -- see e.g.~\cite[Section 1.1]{BP24} for the required background and definitions. Uniform transience implies that there exists $\eps > 0$ such that $\mathrm{Var}(\varphi(x))\leq 1/\eps$ for every $x\in\tilde{V}$.

Fix $o\in V$ and let $\tilde{C}$ be the cluster of $o$ in the percolation model induced by the excursion set $\{\varphi\geq 0\}:=\{x\in \tilde{V}:~\varphi(x)\geq 0\}$. Given every edge $e$ of $G$, we associate the corresponding mid-edge $\tilde{e}$ in $\tilde{G}$, with both endpoints of degree $2$. For a subset $\Pi$ of edges in $G$, we denote by $\tilde{\Pi}$ the associated set of mid-edges in $\tilde{G}$. 
We claim that there exists $c=c(\eps)>0$, depending only on $\varepsilon$, such that for every $\Pi\in\mathcal{Q}_n(o)$,
    \begin{equation}\label{eq:GFF_bound}
        \mathbb{P} ( \partial_\infty \tilde{C}=\tilde{\Pi}) \geq c^n.
    \end{equation}
Similarly to the previous sections, Theorem~\ref{thm:transient->cutset} follows readily from \eqref{eq:GFF_bound}.

    We now proceed to prove \eqref{eq:GFF_bound}.
    Enumerate $\tilde{\Pi}$ by $\tilde{e}_i=\{x_i,y_i\}$, $1\leq i\leq n$, where $x_i$ and $y_i$ are the inner and outer endpoints, respectively.
    We first observe that, for some constant $c_1=c_1(\eps)>0$,
    \begin{equation}\label{eq:exp_closed_GFF}
        \mathbb{P}(\varphi(y_i)\in [-2,-1] \text{ and } \varphi(x_i)\in [1,2]~~\forall~ 0\leq i\leq n)\geq c_1^n.
    \end{equation}
    Indeed, this follows by successively demanding the desired event at each vertex. Here we use the Markov property of the GFF (see \cite[Theorem 1.10]{BP24}) and the fact that the conditional variance of the next vertex given the previous ones is between $1/2$ (since they have degree $2$) and $1/\eps$, while the conditional mean remains bounded between $-2$ and $2$.

    Let $\mathcal{F}$ be the event in \eqref{eq:exp_closed_GFF} and $A$ be the component of $o$ in  $(\tilde{V},\tilde{E}\setminus \tilde{\Pi})$. Notice that 
    \begin{equation*}
        \mathbb{P}(\partial_\infty \tilde{C}=\tilde{\Pi})\geq \mathbb{P}(\mathcal{F}) \mathbb{P}\left( \bigcap_{i=1}^n \{ 0 \xleftrightarrow{ \{\varphi\geq0\}\cap A } x_i \} \,\Big|\, \mathcal{F}  \right) \geq c_1^n \mathbb{P}\left( \bigcap_{i=1}^n \{ 0 \xleftrightarrow{ \{\varphi\geq0\}\cap A } x_i \} \,\Big|\, \mathcal{F} \right).
    \end{equation*}
    By the Markov property, conditionally on $\mathcal{F}$, the process $\{\varphi\geq 0\}\cap A$ stochastically dominates $\{\varphi_A\geq -1\}$, where $\varphi_A$ is the centered GFF on $A$ (i.e.~associated to the random walk on $A$ killed when reaching $\partial A=\{x_1,\dots,x_n\}$). Therefore, it is enough to prove that, for some constant $c_2=c_2(\eps)>0$,
    \begin{equation}\label{eq:exp_connection_GFF}
        \mathbb{P} \left( \bigcap_{i=1}^{n} \{ o \xleftrightarrow{\{\varphi_A\geq -1\}} x_i \} \right) \geq c_2^n.
    \end{equation}
    Indeed, since the GFF is positively associated (see \cite[Theorem 3.38]{BP24}), the desired inequality \eqref{eq:exp_connection_GFF} follows readily from Proposition~\ref{prop:fkg} and the following inequality
    \begin{equation}\label{eq:GFF_percolates}
        \mathbb{P}( u \xleftrightarrow{\{\varphi_A\geq -1\}} \partial A)\geq \mathbb{E}(\text{sgn}(\varphi_A(u)+1) )\geq c_3,
    \end{equation}
    for some constant $c_3=c_3(\eps)>0$. The latter follows easily from the Markov property of the GFF. Indeed, let $\mathcal{S}$ be the union of all clusters of $\{\varphi_A\geq -1\}$ intersecting $\partial A$ and note that its closure $\overline{\mathcal{S}}$ (i.e.~the union of $S$ with its neighbours) is a stopping set. Clearly, one has $\text{sgn}(\varphi_A(u)+1)=1$ almost surely on the event $\mathcal{G}:=\{u \xleftrightarrow{\{\varphi_A\geq -1\}} \partial A\}=\{u\in \mathcal{S}\}$. On the complementary event $\mathcal{G}^c$ and conditionally on the field on $\overline{\mathcal{S}}$, the Markov property implies that we have a GFF on $A\setminus \mathcal{S}$ with boundary conditions $<-1$. In particular, $\text{sgn}(\varphi_A(u)+1)$ has a negative conditional expectation on $\mathcal{G}^c$. These observations readily imply the first inequality of \eqref{eq:GFF_percolates}. The second inequality follows from the fact that the variance of $\varphi_A(u)$ is at most $1/\eps$.

\end{document}